\documentclass[leqno]{amsart}
\usepackage{amsmath, amsthm, amssymb, amstext}
\usepackage{scalerel,stackengine}
\usepackage{mathrsfs}

\usepackage[left=3.5cm,right=3.5cm,top=1.5cm,bottom=1.5cm,includeheadfoot]{geometry}
\usepackage{hyperref,xcolor}
\hypersetup{
 pdfborder={0 0 0},
 colorlinks,
}
\usepackage{color}

\usepackage{todonotes}
\usepackage{enumitem}
\allowdisplaybreaks

\newtheorem{theorem}{Theorem}

\newtheorem{lemma}[theorem]{Lemma}
\newtheorem{proposition}[theorem]{Proposition}

\newcommand{\N}{\mathbb{N}}

\numberwithin{theorem}{section} \numberwithin{equation}{section}

\title[Generalized Telegraph equation with fractional $p(x)$-Laplacian]{Generalized Telegraph equation with fractional $p(x)$-Laplacian}

\author[ J. Vanterler da C. Sousa $^{*}$, Mbarki Lamine and Leandro S. Tavares]{J. Vanterler da C. Sousa $^{*}$, Mbarki Lamine and Leandro S. Tavares}

\address[J. Vanterler da C. Sousa ]{Aerospace Engineering, PPGEA-UEMA, Department of Mathematics, DEMATI-UEMA, 
São Luís, MA 65054, Brazil.}
\email{\tt vanterler@ime.unicamp.br}

\address[Mbarki Lamine] {Department of Mathematics, Faculty of sciences Tunis, University of Tunis el Manar, Tunisia}
\email{\tt mbarki.lamine2016@gmail.com, lamine.mbarki@fst.utm.tn}

\address[Leandro S. Tavares]{Center of  Sciences and   Technology, Federal University of Cariri,   63048-080, Juazeiro do Norte/CE, Brazil}
\email{\tt leandrolstav@gmail.com}


\begin{document}
 \begin{abstract} The purpose of this  paper is devoted to \textcolor{red}{discussing} the existence of solutions for a generalized fractional telegraph equation involving a class of $\psi$-Hilfer fractional with $p(x)$-Laplacian differential equation.
\end{abstract}
\maketitle

\section{Introduction and motivation}

\textcolor{red}{Over the last few decades  problems} involving wave equations have been paid considerable attention due to \textcolor{red}{their} utilization \textcolor{red}{in modeling mathematical physics problems} and view of their applications.

The study of quasilinear wave equations involving the $p-$Laplacian operator in the principal part and viscosity damping \textcolor{red}{has emerged from} the nonlinear Voigt model of longitudinal motion of a rod made \textcolor{red}{from  viscoelastic} material \cite{13}. Specifically, for the so-called Ludwick materials, it can be shown that they obey the following equations, under the effect of an external force $f$. For the Euler rod
\begin{equation*}
    \rho \frac{\partial^{2} u}{\partial x^{2}} = K \frac{\partial}{\partial x} \left(\left|\frac{\partial u}{\partial x} \right|^{p-1} \frac{\partial u}{\partial x}\right)+f.
\end{equation*}
And for the Euler bean 
\begin{equation*}
    \rho A \frac{\partial^{2} u}{\partial x^{2}} =  \frac{\partial^{2}}{\partial x^{2}} \left(K I_{n}\left|\frac{\partial^{2} u}{\partial x^{2}} \right|^{p-1} \frac{\partial^{2} u}{\partial x^{2}}\right)+f,
\end{equation*}
where $u=u(x,t)$ means the displacement at the time $t$ and the space coordinate $x$, $\rho$ represents the density of the material, $K$ is the engineering constant, $p$ is the strain-hardening exponent, $A$ is the cross-sectional area and $I_{n}$ designates the second moment of the cross-section for the material.

Henceforth, It is well known that an improvement in the convergence rate \textcolor{red}{can be achieved} by considering the corresponding second order damped problem, \cite{3,6,10}. For instance, the \textcolor{red}{second-order} damped problem naturally appears in \textcolor{red}{modeling} mechanical systems. Moreover, \textcolor{red}{the} study of problems involving the $p$-Laplacian operator  is a \textcolor{red}{well-consolidated} and rich topic of PDEs due to its well-founded and \textcolor{red}{highly-impact} results, see for instance \cite{Li,Yuksekkaya,Shahrouzi,Messoaudi2,Guo,Li2,Yang1} and the references therein. \textcolor{red}{As a result}, there are numerous works that address questions about the telegraph equation, viscoelastic hyperbolic equation, viscoelastic wave \textcolor{red}{equation,} and wave equation \cite{Majee,Pei1,Antontsev,Acerbi}, among others\textcolor{red}{. These are} elaborated via Dirichlet problems and variational techniques. Also, there is a wide literature about $p$-Laplacian problems involving fractional operators  \cite{Boudjeriou1,Boudjeriou2,Boudjeriou3,Liao,Jiang}.

\textcolor{red}{On the other hand,} scientists have been considering Partial Differential Equations involving variable exponents due to their applicability in several relevant models. One of the main applications of such equations \textcolor{red}{has} appeared in the study of electrorheological fluids. As mentioned in \cite{rare}, the study of these fluids originated with the discovery of Bingham fluids, which spontaneously stop moving. Indeed, in the classical reference \cite{winslow}, W. Winslow presented one of the main properties of electrorheological fluids, which is the emergence of parallel and string-like formations in the fluid, under the influence of an electrical field. This phenomenon is known in the literature as the Winslow effect. Furthermore, the electrical field can increase the fluid's viscosity by up to five orders of magnitude, as stated in \cite{rare}. Additionally, NASA laboratories have conducted several studies on electrorheological fluids, as highlighted in the interesting paper \cite{rad}.

There is a growing interest in problems involving  variable exponents from \textcolor{red}{a} mathematical point of view. For example, in the reference \cite{acermingi}, the regularity of solutions for  a stationary \textcolor{red}{system  related to}  electrorheological fluids is \textcolor{red}{proven}. In the paper  \cite{fzz}, a strong maximum principle for a  variable exponent operator  is obtained, which generalizes the classical case of the Laplacian operator. Moreover,   several applications are explored. \textcolor{red}{We also  highlight}  \cite{minrad}, which provides an overview of elliptic variational problems with general  growth conditions and variable exponents. For a complete presentation of the theory of Sobolev spaces in the \textcolor{red}{context of variable exponents} and their applications, \textcolor{red}{refer to  the}  references  \cite{ruzi,rare}.

 In 2011, Stavrakakis-Stylianou \cite{Stavrakakis}, considered the following problem
\begin{eqnarray}\label{eq111}
\left\{
 \begin{array}{ll}
 u_{tt}-div(|\nabla u |^{p(x)-2} \nabla u)-\Delta u_{t}+g(u)=f(x,t), \;\;\ (x,t)\in\Omega\times(0,T):=Q_{T}\\
u(x,t)=0\;\;\quad (x,t)\in \partial\Omega\times(0,T):=\Gamma_{T}\\
u(x,0)=u_{0}(x),\,\,\,u_{t}(x,0)=u_{1}(x)\,\,\,x\in\Omega,
\end{array}
\right.
\end{eqnarray}
where $\Omega\subset\mathbb{R}^{N}$, $\N\geq 1$ is a bounded domain, $\partial\Omega$ is Lipschitz continuous. The equation  (\ref{eq111}) is related  to several physical applications in  viscoelasticity and processes of filtration through source terms. 

In 2017, Messaoudi-Talahmed \cite{Messaoudi}, treated the following problem 
\begin{eqnarray}\label{eq112}
\left\{
 \begin{array}{ll}
 u_{tt}-div(|\nabla u |^{m(x)-2} \nabla u)+\mu u_{t}=|u|^{p(x)-2}u, \;\;\ in\,\,\Omega\times(0,T)\\
u(x,t)=0\;\; on\,\partial\Omega\times(0,T)\\
u(x,0)=u_{0}(x),\,\,\,u_{t}(x,0)=u_{1}(x)\,\,\,in\,\Omega,
\end{array}
\right.
\end{eqnarray}
where $\mu\geq0$. For more details about  (\ref{eq112}),  see \cite{Messaoudi}. In the same year, Messaoudi and Talahmed \cite{Messaoudi23} established the same problem by adding a constant $b>0$ to the right-hand side of the equation.

In 2018, Messaoudi et al. \cite{Messaoudi1}, studied the problem
\begin{eqnarray}\label{eq113}
\left\{
 \begin{array}{ll}
 u_{tt}-div(|\nabla u |^{r(\cdot)-2} \nabla u)+|u_{t}|^{m(\cdot)-2} u_{t}=0, \;\;in\,\,\Omega\times (0,T)\\
u(x,t)=0\;\;\quad on\,\, \partial\Omega\times(0,T)\\
u(x,0)=u_{0}(x),\,\,\,u_{t}(x,0)=u_{1}(x)\,\,in\,\Omega,
\end{array}
\right.
\end{eqnarray}
where the exponents $m(\cdot)$ and $r(\cdot)$ are given measurable functions on $\Omega$. See \cite{Messaoudi1}, for more details about the problem (\ref{eq113}).

On the other hand, when dealing with fractional differential equations, we are considering  a well-established theory with numerous applications and important consequences that are related  real phenomena \cite{Diethelm,Lakshmikantham,Kilbas}. Moreover, it is interesting to consider a general notion of fractional derivative of a function  with respect to another function,  due to the several  definitions of integrals and fractional derivatives. Such questions was recently considered  in Sousa \& Oliveira \cite{Sousa4}, where the authors introduced the $\psi-$Hilfer fractional derivative and provided several examples. We also point out that  there has been a recent  interest in equations with the  $\psi$-Hilfer fractional operator,  see for example \cite{Alsaedi,Sousa,Sousa12,Zuo,Sousa1,Ezati,Fan}.

So, motivated by the problem (\ref{eq111})-(\ref{eq113}), in this paper, we are concerned with the generalized fractional telegraph equation given by 
{\color{red}\begin{equation}\label{Prob1}
\varepsilon u_{tt}-{}^{\bf H}\mathbf{D}_{T}^{\alpha,\beta;\psi}\big( \big|{}^{\bf H}\mathbf{D}^{\alpha,\beta;\psi}_{0^{+}}u\big|^{p(x)-2}\,\,{}^{\bf H}\mathbf{D}^{\alpha,\beta;\psi}_{0^{+}}u \big)+u_{t}=f(x,t);\;\;\quad\forall (x,t)\in \mathbf{Q}_{T}:=\Omega\times (0,T)
\end{equation}}
where $f(x,t)$ is a extremal force, $\varepsilon>0$, ${}^{\bf H}\mathbf{D}^{\alpha,\beta;\psi}_{0^{+}}(\cdot)$ and ${}^{\bf H}\mathbf{D}_{T}^{\alpha,\beta;\psi}(\cdot)$ are the $\psi$-Hilfer partial fractional derivative of order $\frac{1}{p(x)}<\alpha<1$ and of type $0\leq \beta \leq 1$, with the conditions{\color{red}
 \begin{eqnarray}\label{eq1}
\left\{
 \begin{array}{ll}
 u(x,t)=0, \quad (x,t)\in \Gamma_{T}:=\partial\Omega\times (0,T)\\
u(x,0)=u_{0}(x),\;\;\quad\forall x\in \Omega\\
u_{t}(x,0)=u_{1}(x),\;\;\quad\forall x\in \Omega,
\end{array}
\right.
\end{eqnarray}}
where $\Omega=(0,L)$.

Throughout this work, the  functions $p(x), u_{0}, u_{1}$ and $g$ satisfy the following condition:{\color{red}
\begin{eqnarray}\label{eq2}
\left\{
 \begin{array}{ll}
 2\leq p^{-}\leq p(x)\leq p^{+}<\infty\\
u_{0}\in \mathcal{H}_{p(x)}^{\alpha, \beta;\psi}(\Omega),\; u_{1}\in L^{2}(\Omega),\; g(x)\in L^{\frac{p(x)}{p(x)-1}}(\Omega).
\end{array}
\right.
\end{eqnarray}}

When working with problems with variable exponents, the principal difficulty is dealing with non-homogeneity,   which imply that the application of the  Komornik's inequality is not immediately clear. In order to overcome such difficulty,    the problem considered will be studied by means of an alternative method that consists in considering a second-order dynamical system that establishes the strong convergence of the difference norm.

The main result of the present manuscript is given as follow:
\begin{theorem}\label{theorem 1}
Assume that the condition \eqref{eq2} holds. If $u(x,t)$ satisfies \eqref{Prob1} and  \eqref{eq1}, then for $t>0$ there exists $u^{*}(x)$ satisfying
$$\int_{0}^{L}\left|{}^{\bf H}\mathbf{D}^{\alpha,\beta;\psi}_{0^{+}}u(x,t)-{}^{\bf H}\mathbf{D}^{\alpha,\beta;\psi}_{0^{+}}u^{*}(x)\right|^{p(x)}dx\leq \Theta t^{\frac{1}{1-p^{+}}},$$
where $\Theta=\left[c_{2}^{-p^{+}}(p^{+}-1)\right]^{\frac{1}{1-p^{+}}}$,  which is a constant independent of time the variable but  depends only on the coefficient $\varepsilon$. {\color{red} Here $u^{*}$ is the solution the steady state equation.}
\end{theorem}








Note that the result  above is also held for the equations
\begin{equation*}
\varepsilon u_{tt}-{}^{\bf H}\mathbf{D}_{T}^{\alpha,\beta;\psi}\big( \big|{}^{\bf H}\mathbf{D}^{\alpha,\beta;\psi}_{0^{+}}u\big|^{p-2}\,\,{}^{\bf H}\mathbf{D}^{\alpha,\beta;\psi}_{0^{+}}u \big)+u_{t}=f(x,t).
\end{equation*}
and 
\begin{equation}\label{Prob11}
\varepsilon u_{tt}- \frac{\partial}{\partial x}\left( \left|\frac{\partial u}{ \partial x} \right|^{p(x)-2} \frac{\partial u}{\partial x} \right)+u_{t}=f(x,t);
\end{equation}
with the conditions (\ref{eq1}).

 Consequently, all results investigated here are valid for the classic case.

As noted earlier, one of the advantages of the $\psi$-Hilfer fractional operator is that it is possible to  consider  a wide class of cases. Namely, choosing $\beta=1$ and $\psi(t)=t$ yields the problem in the Caputo fractional operator version. Even though, choosing $\beta=0$ and $\psi(t)=t$ yields the problem in the fractional Riemann-Liouville version. Furthermore, the results investigated in this paper  are valid for both of these particular cases, as well as other specific cases.


The rest of the paper is structured as follows: In Section 2, we present some concepts and prove some important ones to discuss the main result of the paper. Finally, in Section 3, we investigate the main contribution of this present manuscript, i.e., the proof of {\bf Theorem \ref{theorem 1}}.

\section{Mathematical background: preliminaries}

Let $\Omega\subset \mathbb{R}^{N}$ be an open set. We denote by $|\Omega|$ the $N$-dimensional Lebesgue measure of $\Omega$. This section is devoted to recall some preliminary definitions and results that we need in the rest of the paper.

For this aim, let us introduce the space
$$
C_+(\overline {\Omega})=\{p\in C(\overline {\Omega};\mathbb{R}):
\inf_{x\in\Omega}p(x)>1\}.
$$

The variable exponent Lebesgue space $L^{p(\cdot)}(\Omega)$ is defined by
\begin{align*}
   \mathscr{L}^{p(\cdot)}(\Omega) = \left\{u:\Omega \rightarrow \mathbb{R} \mbox{ measurable:} \displaystyle\int_\Omega |u(x)|^{p(x)}dx<\infty\right\}.
\end{align*}
$\mathscr{L}^{p(\cdot)}$ is a Banach space when endowed with the Luxemburg norm defined by
\begin{eqnarray*}
    |u|_{p(\cdot)}:=\inf\left\{\mu>0: \int_{\Omega} \left|\frac{u(x)}{\mu} \right|^{p(x)} dx \leq 1\right\}.
\end{eqnarray*}
The variable exponent Lebesgue space $\mathscr{L}^{p(\cdot)}(\Omega)$ is a special case of an Orlicz-Musielak space.

For any Lipschitz continuous function $p: \overline{\Omega}\rightarrow (1,\infty)$, let \cite{passa}
\begin{eqnarray*}
    p^{-}:= \inf_{x\in\Omega} p(x)\,\,and\,\, p^{+}:= \sup_{x\in\Omega} p(x).
\end{eqnarray*}

It is well known (see \cite{Boudjeriou1}) that for each $p_1$, $p_2\in C_+(\overline\Omega)$ such that $p_1 \leq p_2$ in ${\Omega}$,  the embedding 
$\mathscr{L}^{p_2(\cdot)}({\Omega})\hookrightarrow \mathscr{L}^{p_1(\cdot)}({\Omega})$ is continuous. Furthermore, For $p(x)\in C_+(\overline\Omega)$, if we take $p'$ such that  $\frac{1}{p(x)}+\frac{1}{p'(x)}=1$, then the H\"older inequality (see \cite{Boudjeriou1, Boudjeriou2}) is given as follows
$$ \left|\int_{\Omega}uvdx\right|\leq\left(\frac{1}{p^-}+\frac{1}{p'^-}\right)||u||_{p(x)}||v||_{p'(x)}\leq 2 ||u||_{p(x)}||v||_{p'(x)},
$$
for any $u\in \mathscr{L}^{p(x)}(\Omega)$ and  $v\in \mathscr{L}^{p'(x)}(\Omega).$

In addition, the relation between the norm $\|u\|_{p(x)}$ and its modular function $\displaystyle\int_{\Omega}|u(x)|^{p(x)}dx$ is stated as follows
$$||u||_{p(x)} \geq 1\; \Rightarrow\; ||u||^{p^{-}}_{p(x)} \leq \int_{\Omega}|u(x)|^{p(x)}dx \leq ||u||^{p^{+}}_{p(x)}.$$
$$||u||_{p(x)} < 1\; \Rightarrow\; ||u||^{p^{+}}_{p(x)} \leq \int_{\Omega}|u(x)|^{p(x)}dx \leq ||u||^{p^{-}}_{p(x)}.$$

The $\psi$-fractional space is given by \cite{Sousa1}{\color{red}
\begin{equation*}
\mathcal{H}^{\alpha,\beta;\psi}_{p(x)}(\Omega)=\left\{u\in \mathscr{L}^{p(x)}(\Omega)~:~\left\vert^{\rm H}{\bf D}^{\alpha,\beta;\,\psi}_{0+}u\right\vert\in \mathscr{L}^{p(x)}(\Omega)\right\}
\end{equation*}}
with the norm 
\begin{equation*}
||u||=||u||_{\mathcal{H}^{\alpha,\beta;\psi}_{p(x)}(\Omega)}=||u||_{\mathscr{L}^{p(x)}(\Omega)}+\left\|^{\rm H}{\bf{D}}_{0+}^{\alpha,\beta;\psi}u\right\|_{\mathscr{L}^{p(x)}(\Omega)}.
\end{equation*}

Let $\theta=(\theta_{1},\theta_{2},\theta_{3})$, $T=(T_{1},T_{2},T_{3})$ and $\alpha=(\alpha_{1},\alpha_{2},\alpha_{3})$ where $0<\alpha_{1},\alpha_{2},\alpha_{3}<1$ with $\theta_{j}<T_{j}$, for all $j\in \left\{1,2,3 \right\}$. Also put $\Lambda=I_{1}\times I_{2}\times \times I_{3}=[\theta_{1},T_{1}]\times [\theta_{2},T_{2}]\times [\theta_{3},T_{3}]$, where $T_{1},T_{2},T_{3}$ and $\theta_{1},\theta_{2},\theta_{3}$ positive constants. Consider also $\psi(\cdot)$ be an increasing and positive monotone function on $(\theta_{1},T_{1}),(\theta_{2},T_{2}),(\theta_{3},T_{3})$, having a continuous derivative $\psi'(\cdot)$ on $(\theta_{1},T_{1}],(\theta_{2},T_{2}],(\theta_{3},T_{3}]$.

On the other hand, let $u,\psi \in C^{n}(\Lambda)$ two functions such that $\psi$ is increasing and $\psi'(x_{j})\neq 0$ with $x_{j}\in[\theta_{j},T_{j}]$, $j\in \left\{1,2,3 \right\}$. The left and
right-sided $\psi$-Hilfer fractional partial derivative of $3$-variables of $u\in AC^{n}(\Lambda)$ of order $\alpha=(\alpha_{1},\alpha_{2},\alpha_{3})$ $(0<\alpha_{1},\alpha_{2},\alpha_{3}\leq 1)$ and type $\beta=(\beta_{1},\beta_{2},\beta_{3})$ where $0\leq\beta_{1},\beta_{2},\beta_{3}\leq 1$, are defined by \cite{Srivastava,Sousa4,Sousa5}
\begin{equation*}
{^{\mathbf H}{\bf D}}^{\alpha,\beta;\psi}_{\theta}u(x_{1},x_{2},x_{3})= {\bf I}^{\beta(1-\alpha),\psi}_{\theta} \Bigg(\frac{1}{\psi'(x_{1})\psi'(x_{2})\psi'(x_{3})} \Bigg(\frac{\partial^{3}} {\partial x_{1}\partial x_{2}\partial x_{3}}\Bigg) \Bigg) {\bf I}^{(1-\beta)(1-\alpha),\psi}_{\theta} u(x_{1},x_{2},x_{3})
\end{equation*}
and
\begin{equation*}
{^{\mathbf H}{\bf D}}^{\alpha,\beta;\psi}_{T}u(x_{1},x_{2},x_{3})= {\bf I}^{\beta(1-\alpha),\psi}_{T} \Bigg(-\frac{1}{\psi'(x_{1})\psi'(x_{2})\psi'(x_{3})} \Bigg(\frac{\partial^{3}} {\partial x_{1}\partial x_{2}\partial x_{3}}\Bigg) \Bigg) {\bf I}^{(1-\beta)(1-\alpha),\psi}_{T} u(x_{1},x_{2},x_{3}),
\end{equation*}
where ${\bf I}^{\alpha,\psi}_{\theta} u(x_{1},x_{2},x_{3})$ and ${\bf I}^{\alpha,\psi}_{T} u(x_{1},x_{2},x_{3})$ there are the $\psi$-Riemann-Liouville fractional integrals of $u\in \mathscr{L}^{1}(\Lambda)$ of order $\alpha$ $(0<\alpha<1)$ given by \cite{Srivastava,Sousa4,Sousa5}    
\begin{eqnarray*}
    {\bf I}^{\alpha,\psi}_{\theta} u(x_{1},x_{2},x_{3})=\dfrac{1}{\Gamma(\alpha_{1})\Gamma(\alpha_{2})\Gamma(\alpha_{3})} \int_{\theta_{1}}^{x_{1}}
    \int_{\theta_{2}}^{x_{2}}
    \int_{\theta_{3}}^{x_{3}}
    \psi'(s_{1})\psi'(s_{2})\psi'(s_{3})
    (\psi(x_{1})- \psi(s_{1}))^{\alpha_{1}-1}\notag\\
    \times
    (\psi(x_{2})- \psi(s_{2}))^{\alpha_{2}-1}
    (\psi(x_{3})- \psi(s_{3}))^{\alpha_{3}-1}
    u(s_{1},s_{2},s_{3}) ds_{3}ds_{2}ds_{1},
\end{eqnarray*}
to $\theta_{1}<s_{1}<x_{1}, \theta_{2}<s_{2}<x_{2}, \theta_{3}<s_{3}<x_{3}$ and
\begin{eqnarray*}
    {\bf I}^{\alpha,\psi}_{T} u(x_{1},x_{2},x_{3})=\dfrac{1}{\Gamma(\alpha_{1})\Gamma(\alpha_{2})\Gamma(\alpha_{3})} \int_{x_{1}}^{T_{1}}
    \int_{x_{2}}^{T_{2}}
    \int_{x_{3}}^{T_{3}}
    \psi'(s_{1})\psi'(s_{2})\psi'(s_{3})
    (\psi(s_{1})-\psi(x_{1}))^{\alpha_{1}-1}\notag\\
    \times
    (\psi(s_{2})-\psi(x_{2}))^{\alpha_{2}-1}
    (\psi(s_{3})-\psi(x_{3}))^{\alpha_{3}-1}
    u(s_{1},s_{2},s_{3}) ds_{3}ds_{2}ds_{1},
\end{eqnarray*}
with $x_{1}<s_{1}<T_{1}, x_{2}<s_{2}<T_{2}, x_{3}<s_{3}<T_{3}$, $x_{1}\in[\theta_{1},T_{1}]$, $x_{2}\in[\theta_{2},T_{2}]$ and $x_{3}\in[\theta_{3},T_{3}]$. For a study of $N$-variables, see \cite{Srivastava}.

Let $\theta=(\theta_{1},\theta_{2}, \theta_{3})$, $T=(T_{1},T_{2},T_{3})$ and $\alpha=(\alpha_{1},\alpha_{2},\alpha_{3})$. The relation
\begin{eqnarray}\label{eq.218}
&&\int_{a}^{b}\int_{c}^{d} \int_{e}^{f}\left( {\bf I}_{\theta}^{\alpha ;\psi }\varphi\left( \xi_{1},\xi_{2},\xi_{3}\right) \right) \phi\left( \xi_{1},\xi_{2},\xi_{3}\right) d\xi_{3} d\xi_{2} d\xi_{1}\notag\\&=&\int_{a}^{b}\int_{c}^{d}\int_{e}^{f}\varphi\left( \xi_{1},\xi_{2},\xi_{3}\right) \psi ^{\prime }\left( \xi_{1}\right) \psi'(\xi_{2}) \psi'(\xi_{3}) {\bf I}_{T}^{\alpha ;\psi }\left( \frac{\phi\left( \xi_{1},\xi_{2},\xi_{3}\right) }{\psi ^{\prime }\left( \xi_{1}\right)\psi'(\xi_{2}) \psi'(\xi_{3}) }\right) d \xi_{3} d\xi_{2} d\xi_{1}
\end{eqnarray}%
is valid.

Whereas, let  $\psi(\cdot)$ be an increasing and positive monotone function on $[a,b]\times [c,d]\times [e,f]$, having a continuous derivative $\psi'(\cdot)\neq 0$ on $(a,b)\times(c,d)\times(e,f)$. If $0<\alpha=(\alpha_{1},\alpha_{2},\alpha_{3}) <1$ and $0\leq \beta=(\beta_{1},\beta_{2},\beta_{3}) \leq 1$, then
\begin{eqnarray}\label{mera}
&&\int_{a}^{b}\int_{c}^{d} \int_{e}^{f}\left( ^{{\bf H}}\mathbf{D}_{\theta}^{\alpha,\beta ;\psi }\varphi\left( \xi_{1},\xi_{2},\xi_{3}\right) \right) \phi\left( \xi_{1},\xi_{2},\xi_{3}\right) d\xi_{2} d\xi_{1}\notag\\&&=\int_{a}^{b}\int_{c}^{d}\int_{e}^{f}\varphi\left( \xi_{1},\xi_{2},\xi_{3}\right) \psi ^{\prime }\left(\xi_{1}\right) \psi ^{\prime }\left(\xi_{2}\right) \psi'(\xi_{3}) \text{ }^{{\bf H}}\mathbf{D}_{T}^{\alpha ,\beta ;\psi }\left( \frac{\phi\left( \xi_{1},\xi_{2},\xi_{3}\right) }{\psi ^{\prime }\left( \xi_{1}\right)\psi ^{\prime }\left( \xi_{2}\right) \psi'(\xi_{3}) }  \right) d \xi_{3} d\xi_{2}d\xi_{1}
\end{eqnarray}
for any $\varphi\in AC^{1}$ and $\phi\in C^{1}$ satisfying the boundary conditions $\varphi\left( a,c,e\right)=0=\varphi\left( b,d,f\right)$.

For the proof of the next result, we will use the notation $^{\rm H}{\bf D}^{\alpha,\beta;,\psi}_{0+} (\cdot) =  \partial^{\alpha,\beta;\psi}{x_{j}}(\cdot)$, where it is the $\psi$-Hilfer fractional derivative with respect to the variable $x_j$, with $j=1,2,3,...,n$.

{\color{red} Note that for $u\in \mathcal{H}^{\alpha,\beta;\psi}_{p(x)}(\Omega)$, $\partial^{\alpha,\beta;\psi}{x_{j}}u$ denotes the $j$-th weak fractional derivative of $u,$ which is given by}
\begin{eqnarray*}
    \int_{\Omega} u\,\, \partial^{\alpha,\beta;\psi}_{x_{j}} \phi dx = -\int_{\Omega} \partial^{\alpha,\beta;\psi}_{x_{j}} u\, \phi dx,
\end{eqnarray*}
for all $ \phi\in C^{\infty}_{0}(\Omega)$.

\begin{proposition}\label{bana} {\rm\cite{Sousa1,Srivastava}} The space $\mathscr{L}^{p(x)}(\Omega)$ is separable and reflexive Banach spaces.
\end{proposition}

\begin{proposition}\label{bana} The space $\mathcal{H}^{\alpha,\beta;\,\psi}_{p(x)}(\Omega)$ are separable and reflexive Banach spaces is $p^{-}>1$.
\end{proposition}
\begin{proof} Let $\left\{ u_{m}\right\}\subset  \mathcal{H}^{\alpha,\beta;\psi}_{p(x)}(\Omega)$ a Cauchy sequence. Then $\left\{ u_{m}\right\}$ and $\left\{ \partial^{\alpha,\beta;\psi}_{x_{j}} u_{m}\right\}$, $j=1,2,...,N$ are Cauchy sequences in $\mathscr{L}^{p(x)}(\Omega)$. Since $\mathscr{L}^{p(x)}(\Omega)$ is a Banach space, there are $u, w_{j}\in \mathscr{L}^{p(x)}(\Omega) $ such that{\color{red}
\begin{eqnarray}\label{eq2.1}
    u_{n}\rightarrow u\,\,and\,\, \partial^{\alpha,\beta;\psi}_{x_{j}} u_{n}\rightarrow\, w_{j}\,\,in\,\,\mathscr{L}^{p(x)}(\Omega),
\end{eqnarray}}
when $n\rightarrow\infty$ with $j=1,2,...,N$.

Using Holder's inequality, we have
\begin{eqnarray}\label{eq2.2}
    \int_{\Omega} (u_{n}-u)\, \partial^{\alpha,\beta;\psi}_{x_{j}} \phi dx \leq C ||u_{n}-u||_{p(x)} \left\|\partial^{\alpha,\beta;\psi}_{x_{j}} \phi \right\|
\end{eqnarray}
to $\phi\in C^{\infty}_{0}(\Omega)$, where $C=\dfrac{1}{p^{-}}+\dfrac{1}{q^{-}}$. It follows from Eq.(\ref{eq2.1}) and Eq.(\ref{eq2.2}) that
\begin{eqnarray}\label{eq2.3}
    \int_{\Omega} u_{n} \partial^{\alpha,\beta;\psi}_{x_{j}} \phi dx \rightarrow \int_{\Omega} u\, \partial^{\alpha,\beta;\psi}_{x_{j}} \phi dx
\end{eqnarray}
to $\phi\in C^{\infty}_{0}(\Omega)$, when $n\rightarrow\infty$. 

Similarly, 
\begin{eqnarray}\label{eq2.4}
    \int_{\Omega} \partial^{\alpha,\beta;\psi}_{x_{j}} u_{n}\,\phi dx \rightarrow \int_{\Omega} w_{j}\, \phi dx
\end{eqnarray}
to $\phi \in C^{\infty}_{0}(\Omega)$, when $n\rightarrow\infty$.

Since that
\begin{eqnarray}\label{eq2.5}
    \int_{\Omega} u_{n} \partial^{\alpha,\beta;\psi}_{x_{j}} \phi dx =- \int_{\Omega} \partial^{\alpha,\beta;\psi}_{x_{j}} u_{n}\,\phi dx
\end{eqnarray}
for all $\phi\in C^{\infty}_{0}(\Omega)$. Passing to the limit in (\ref{eq2.5}) of $n\rightarrow\infty$ and using Eq.(\ref{eq2.3}) and Eq.(\ref{eq2.4}), yields
\begin{eqnarray}\label{eq2.6}
    \int_{\Omega}u \partial^{\alpha,\beta;\psi}_{x_{j}} \phi dx =-\int_{\Omega} w_{j}\, \phi dx
\end{eqnarray}
to $\phi \in C^{\infty}_{0}(\Omega)$.

Using the Du Bois-Reymond lemma in (\ref{eq2.6}), we conclude that $u\in \mathcal{H}^{\alpha,\beta;\psi}_{p(x)}(\Omega)$ and $w_{j}=\partial^{\alpha,\beta;\psi}_{x_{j}}u$, $j=1,2,...,N$. So, using Eq.(\ref{eq2.1}), we have
\begin{eqnarray*}
    \left\|u_{n}-u \right\|_{*}= \left\|u_{n}-u \right\|_{p(x)}+\sum_{j=1}^{N} \left\| \partial^{\alpha,\beta;\psi}_{x_{j}} u_{n}-\,\partial^{\alpha,\beta;\psi}_{x_{j}} u \right\|_{p(x)}\rightarrow 0
\end{eqnarray*}
when $n\rightarrow 0$. Therefore, $\mathcal{H}^{\alpha,\beta;\psi}_{p(x)}(\Omega)$ is a Banach space.

Now, let's prove that it is separable and reflexive. Note that the space $E= L^{p(x)}(\Omega)\times\cdots \times L^{p(x)}(\Omega)$, $(n+1)-$times, equipped with the norm $||\cdot||$ is reflective and separable. Define the linear operator $T:\mathcal{H}^{\alpha,\beta;\psi}_{p(x)} (\Omega)\rightarrow E$ to $T(u)= \left(u, \partial^{\alpha,\beta;\psi}_{x} u\right)$. Note that $||Tu||=||u||$. From this last equality we conclude that $T \left( \mathcal{H}^{\alpha,\beta;\psi}_{p(x)}(\Omega)\right)$ is a closed subspace of $E$. From the Proposition III.17 and III. 22 of \cite{brezis} we have $T \left( \mathcal{H}^{\alpha,\beta;\psi}_{p(x)}(\Omega)\right)$ it is reflexive and separable. Therefore, $\mathcal{H}^{\alpha,\beta;\psi}_{p(x)}(\Omega)$ is reflexive and separable. 
\end{proof}

{\color{red} To understand our technique easier, we only consider $f(x,t)=g(x)=A p_{0}(x)$  in which $p_{0}(x)$ represents the static pressure.}

{\color{red} Consider $\varepsilon u_{tt}=\mathcal{M}+g(x,t)$ where $\mathcal{M}:= {}^{\bf H}\mathbf{D}_{T}^{\alpha,\beta;\psi}\big( \big|{}^{\bf H}\mathbf{D}^{\alpha,\beta;\psi}_{0^{+}}u\big|^{p(x)-2}\,\,{}^{\bf H}\mathbf{D}^{\alpha,\beta;\psi}_{0^{+}}u \big)-u_{t}$.

The Galerkin approximations of solutions to problem (\ref{Prob1}) with the condition (\ref{eq1}) are sought in the form
\begin{equation}\label{25}
    u^{(n)}= \sum_{k=1}^{n}u_{k}(t)\psi_{k}(t),\,\, u_{k}=\left(u(x,t),\psi_{k}(t) \right)_{\mathcal{H}_{2}^{\alpha,\beta;\psi}(\Omega)}.
\end{equation}

We assume also
\begin{equation}\label{26}
    u_{1}^{(n)}\rightarrow u_{1},\,\,strongly\,\,in\,\, L^{2}(\Omega),\,\,u_{0}^{(n)}\rightarrow u_{0},\,\,strongly\,\,in\,\,\mathcal{H}^{\alpha,\beta;\psi}_{2}(\Omega).
\end{equation}

The coefficient $u_{k}(t)$ are defined from the relation
\begin{equation}\label{27}
    \int_{0}^{L} \left(\varepsilon u_{tt}^{(n)}-\mathcal{M} u^{(n)}-g\right)\psi_{k}=0,\,\,k=1,2,...,n.
\end{equation}

Last equalities and the initial conditions lead us to the Cauchy problem for the system of $n$ ordinary differential equations of the second order for the coefficients $u_{k}(t)$ 
\begin{eqnarray}\label{28}
    u''_{k}= \mathcal{G}_{k}(t,u(t),..., u_{n}(t))
\end{eqnarray}
and
\begin{eqnarray}\label{29}
    u_{k}(0)=\int_{0}^{L} u_{0} \psi_{k},\,\,u'_{k}(0)=\int_{0}^{L} u_{1}\psi_{k},\,\,\, k=1,2,...,n
\end{eqnarray}
where
\begin{eqnarray*}
    \mathcal{G}_{k}=\int_{0}^{L} \left(\left|{}^{\bf H}\mathbf{D}^{\alpha,\beta;\psi}_{0^{+}} u^{(n)}(x)\right|^{p(x)-2}\,\,{}^{\bf H}\mathbf{D}^{\alpha,\beta;\psi}_{0^{+}} u^{(n)} \right)\,{}^{\bf H}\mathbf{D}^{\alpha,\beta;\psi}_{0^{+}}\psi_{k}dx+ \int_{0}^{L} (g\psi_{k}- u_{t}^{(n)}\psi_{k})dx.
\end{eqnarray*}

Using the Peano’s Theorem, for every finite $n$ the problem (\ref{28}), (\ref{29}) has a solution $u_{k}(t), k=1,...,n$ on an interval $(0,T_{n})$ for each $n$. The estimates below allow one to take $T_{n}=T$ for all $n$.

Multiplying each of equations (\ref{28}) by $c'_{k}(t)$ and summing over $k = 1,...,n$, we arrive at the relation
\begin{eqnarray}\label{30}
    &&\frac{d}{dt}\int_{0}^{L} \left(\frac{\varepsilon}{2} \left|u_{t}^{(n)}\right|^{2}+ \frac{\left|{}^{\bf H}\mathbf{D}^{\alpha,\beta;\psi}_{0^{+}} u^{(n)} \right|^{p(x)}}{p(x)} - u^{(n)} g(x)\right)dx+\int_{0}^{L} \int_{0}^{L} \left|u^{(n)}\right|^{2} dx= \int_{0}^{L} g(x) u_{t}^{(n)}dx.\notag\\
\end{eqnarray}

Omitting the index $n$ for simplicity, we define the energy functional associated to the problem \eqref{Prob1} as follows
\begin{equation}\label{eq3}
\mathbb{E}_{p(x)}^{\alpha, \beta;\psi}(t)=\frac{\varepsilon}{2}\int_{0}^{L}\left|u_{t}\right|^{2}dx+\int_{0}^{L}\frac{1}{p(x)}\left|{}^{\bf H}\mathbf{D}^{\alpha,\beta;\psi}_{0^{+}}u\right|^{p(x)}dx-\int_{0}^{L}ug(x)dx. 
\end{equation}

Note that, the Eq.(\ref{30}) can write the following form
\begin{eqnarray}\label{32}
    \left( \mathbb{E}_{p(x)}^{\alpha, \beta;\psi}(t)\right)'+ \int_{0}^{L}|u_t|^{2} dx= \int_{0}^{L} g(x) u_{t}dx.
\end{eqnarray}

Taking $g=0$ and using Eq.(\ref{32}), we have the following result:}

{\color{red}
\begin{lemma}\label{lemma1}
Assume that the condition \eqref{eq2} hold and $u\in L^{\infty}(0,T; \mathcal{H}_{p(x)}^{\alpha,\beta;\psi}(\Omega))$ is a solution of the problem \eqref{Prob1} such that $u_{t}\in L^{\infty}(0,T;L^{2}(\Omega))$ and $u_{tt}\in L^{\infty}\left(0,T;\mathcal{H}_{\frac{p(x)}{p(x)-1}}^{\alpha,\beta;\psi}(\Omega)\right)$. Then the energy functional $\mathbb{E}_{p(x)}^{\alpha, \beta;\psi}(t)$ satisfies the following inequality
\begin{equation}\label{eq4}
\mathbb{E}_{p(x)}^{\alpha, \beta;\psi}(t)+\int_{0}^{t}\int_{0}^{L}\left|u_{s}\right|^{2}dxds\leq \mathbb{E}_{p(x)}^{\alpha, \beta;\psi}(0),\; 0\leq t\leq T.
\end{equation}
\end{lemma}}

Obviously, with the help of the Eq.\eqref{eq4} and the Poincar\'e inequality \cite{Sousa,Sousa12},
\begin{equation}\label{A1}
\|u\|_{p(\cdot)}\leq C \left\|{}^{\bf H}\mathbf{D}^{\alpha,\beta;\psi}_{0^{+}}u\right\|_{p(\cdot)},\; \text{with}\quad u\in \mathcal{H}_{p(x)}^{\alpha,\beta;\psi}
\end{equation}
we have the following result:

\begin{proposition}\label{proposition1} If all the condition of the {\bf Lemma \ref{lemma1}} are satisfied, then there exists $\Theta_{0}>0$ (constant) such that for $t>0$, we have
\begin{equation}\label{eq6}
\varepsilon \int_{0}^{L}|u_t|^{2}dx+\int_{0}^{L}\frac{1}{p(x)}\left|{}^{\bf H}\mathbf{D}^{\alpha,\beta;\psi}_{0^{+}}u\right|^{p(x)}dx+2\int_{0}^{t}\int_{0}^{L}|u_\tau|^{2}dxd\tau\leq \Theta_{0}.
\end{equation}
\end{proposition}

\begin{proof} Using the Young inequality, there exists $C>0$ (constant), such that
\begin{align}\label{eq7}
\mathbb{E}_{p(x)}^{\alpha,\beta;\psi}(t)&\geq \frac{\varepsilon}{2} \int_{0}^{L}|u_t|^{2}dx+\int_{0}^{L}\frac{1}{p(x)}\left|{}^{\bf H}\mathbf{D}^{\alpha,\beta;\psi}_{0^{+}}u\right|^{p(x)}dx-C\|g\|_{p'(\cdot)}\left\|{}^{\bf H}\mathbf{D}^{\alpha,\beta;\psi}_{0^{+}}u\right\|_{p(\cdot)}\nonumber\\
&\geq \frac{\varepsilon}{2} \int_{0}^{L}|u_t|^{2}dx+\int_{0}^{L}\frac{1}{p(x)}\left|{}^{\bf H}\mathbf{D}^{\alpha,\beta;\psi}_{0^{+}}u\right|^{p(x)}dx-C\|g\|_{p'(\cdot)}\nonumber\\
&\times
{\color{red}max \left\{\left(\int_{0}^{L} \left|{}^{\bf H}\mathbf{D}^{\alpha,\beta;\psi}_{0^{+}}u\right|^{p(x)} dx\right)^{1/p^{+}},\left(\int_{0}^{L} \left|{}^{\bf H}\mathbf{D}^{\alpha,\beta;\psi}_{0^{+}}u\right|^{p(x)}dx \right)^{1/p^{-}}\right\}}
\nonumber\\
&\geq \frac{\varepsilon}{2} \int_{0}^{L}|u_t|^{2}dx+\int_{0}^{L}\left(\frac{1}{p(x)}-\frac{1}{2p(x)}\right)\left|{}^{\bf H}\mathbf{D}^{\alpha,\beta;\psi}_{0^{+}}u\right|^{p(x)}dx-
C max \left\{||g||^{p'^{+}}_{p'}, ||g||^{p'^{-}}_{p'}\right\}
\nonumber\\
&\geq \frac{\varepsilon}{2} \int_{0}^{L}|u_t|^{2}dx+\int_{0}^{L}\frac{1}{2p(x)}\left|{}^{\bf H}\mathbf{D}^{\alpha,\beta;\psi}_{0^{+}}u\right|^{p(x)}dx-
C max \left\{||g||^{p'^{+}}_{p'(\cdot)}, ||g||^{p'^{-}}_{p'(\cdot)}\right\}
\end{align}
where $p'(x)=\frac{p(x)}{p(x)-1}$, $p'^{_{+}}=\underset{x\in\Omega}{max}\, p'(x)$, $p'^{_{-}}=\underset{x\in\Omega}{min}\, p'(x)$. On the other hand, from the inequalities (\ref{eq4}) and (\ref{eq7}), one has
{\color{red}\begin{eqnarray*}
    &&\varepsilon\int_{0}^{L} |u_{t}|^{2}dx+\int_{0}^{L}\frac{1}{p(x)} \left|^{\rm H}{\bf D}^{\alpha,\beta;\psi}_{0+} u\right|^{p(x)} dx+ 2\int_{0}^{t}\int_{0}^{L}|u_{\tau}|^{2}dx d\tau\notag\\ &\leq&
2 \left|\mathbb{E}^{\alpha,\beta;\psi}_{p(x)}(0)\right|+2C\,max\left\{||g||_{p'(\cdot)}^{p'{_{+}}}, ||g||_{p'(\cdot)}^{p'{_{-}}} \right\}:=\Theta_{0},
\end{eqnarray*}}
where $p'(\cdot)$ is the conjugate of $p(\cdot)$.
\end{proof}

Next, we will obtain the existence of solutions to the stationary problem.

\begin{lemma}\label{lemma2}
If $g(x)\in L^{\frac{p(x)}{p(x)-1}}(\Omega)$, then $u^{*}(x)$ is the minimizer of the following functional
\begin{equation}\label{eq8}
{\mathcal{I}}_{p(x)}^{\alpha,\beta;\psi}(v)=\int_{0}^{L}\frac{1}{p(x)}\left|{}^{\bf H}\mathbf{D}^{\alpha,\beta;\psi}_{0^{+}}v\right|^{p(x)}dx-\int_{0}^{L}g(x)v dx.
\end{equation}
Moreover, it satisfies the following fractional Euler-Lagrange equation in the sense of distribution
{\color{red}\begin{eqnarray}\label{eq9}
\left\{
 \begin{array}{ll}
 -\,{}^{\bf H}\mathbf{D}_{T}^{\alpha,\beta;\psi}\left( \left|{}^{\bf H}\mathbf{D}^{\alpha,\beta;\psi}_{0^{+}}u\right|^{p(x)-2}\,\,{}^{\bf H}\mathbf{D}^{\alpha,\beta;\psi}_{0^{+}}u^{*} \right)=g(x),\; x\in \Omega\\
u^{*}(x)=0,\; x\in \partial{\Omega}.
\end{array}
\right.
\end{eqnarray}}
\end{lemma}

\section{Main results}

Before attacking the main result of the paper, it is necessary to investigate other results that are of paramount importance for the proof of {\bf Theorem \ref{theorem 1}}.

\begin{lemma}\label{lemma3} Let $u$ and $u^{*}$ be solutions of the {\rm Eq.\eqref{Prob1}} and {\rm Eq.\eqref{eq8}} respectively, then we have
$$\lim_{t\to \infty+} {\mathcal{I}}_{p(x)}^{\alpha,\beta;\psi}(u(\cdot,t))={\mathcal{I}}_{p(x)}^{\alpha,\beta;\psi}(u^{*}(\cdot)).$$
\end{lemma}

\begin{proof} For any $t>0$, $u(\cdot,t)\in \mathcal{H}_{p(x)}^{\alpha,\beta;\psi}(\Omega)$ and the {\bf Lemma \ref{lemma2}}, one has
\begin{equation*}
{\mathcal{I}}_{p(x)}^{\alpha,\beta;\psi}(u^{*}(\cdot))\leq {\mathcal{I}}_{p(x)}^{\alpha,\beta;\psi}(u(\cdot,t))
\end{equation*}
implies that
\begin{eqnarray}\label{eq10}
    {\mathcal{I}}_{p(x)}^{\alpha,\beta;\psi}(u^{*}(\cdot))\leq \liminf_{t\to \infty+}{\mathcal{I}}_{p(x)}^{\alpha,\beta;\psi}(u(\cdot,t)).
\end{eqnarray}

Conversely, we want to claim that $\limsup_{t\to \infty}{\mathcal{I}}_{p(x)}^{\alpha,\beta;\psi}(u(\cdot,t))\leq  {\mathcal{I}}_{p(x)}^{\alpha,\beta;\psi}(u^{*}(\cdot))$. To prove this inequality, set
\begin{equation*}
\phi(t):=\frac{1}{2}\int_{0}^{L}\big(u(x,t)-u^{*}(x)\big)^{2}dx.
\end{equation*}

Hence, in this sense, yields
\begin{eqnarray}\label{eq11}
\varepsilon \phi''(t)+\phi'(t)&=&-\int_{0}^{L}\big|{}^{\bf H}\mathbf{D}^{\alpha,\beta;\psi}_{0^{+}}u\big|^{p(x)-2}\,\,{}^{\bf H}\mathbf{D}^{\alpha,\beta;\psi}_{0^{+}}u \,\,{}^{\bf H}\mathbf{D}^{\alpha,\beta;\psi}_{0^{+}}(u-u^{*})dx\notag\\&+& \int_{0}^{L}g(x)(u-u^{*})dx+\varepsilon \int_{0}^{L}|u_t|^{2}dx.
\end{eqnarray}

Using the Young inequality, the monotonicity of the energy functional $\mathbb{E}_{p(x)}^{\alpha,\beta;\psi}(t)$ and the Eq.\eqref{eq11} may be rewritten as
\begin{eqnarray}\label{eq12}
&&\varepsilon \phi''(t)+\phi'(t)\notag\\&\leq& {\color{red}-\int_{0}^{L}\left|{}^{\bf H}\mathbf{D}^{\alpha,\beta;\psi}_{0^{+}}u\right|^{p(x)}dx+ \int_{0}^{L}\frac{p(x)-1}{p(x)}\left|{}^{\bf H}\mathbf{D}^{\alpha,\beta;\psi}_{0^{+}}u\right|^{p(x)}dx+ \int_{0}^{L}\frac{1}{p(x)}\left|{}^{\bf H}\mathbf{D}^{\alpha,\beta;\psi}_{0^{+}}u^{*}\right|^{p(x)}dx}\nonumber\\
&{+}&\int_{0}^{L}g(x)(u-u^{*})dx+ \varepsilon\int_{0}^{L}|u_t|^{2}dx\nonumber\\
&\leq&{\mathcal{I}}_{p(x)}^{\alpha,\beta;\psi}(u^{*}(\cdot))-{\mathcal{I}}_{p(x)}^{\alpha,\beta;\psi}(u(\cdot,t))+\varepsilon\int_{0}^{L}|u_t|^{2}dx\nonumber\\
&{=}&{\mathcal{I}}_{p(x)}^{\alpha,\beta;\psi}(u^{*}(\cdot))+\frac{3\varepsilon}{2}\int_{0}^{L}|u_t|^{2}dx-\mathbb{E}_{p(x)}^{\alpha,\beta;\psi}(t)\nonumber\\
&{\leq}&{\mathcal{I}}_{p(x)}^{\alpha,\beta;\psi}(u^{*}(\cdot))+\frac{3\varepsilon}{2}\int_{0}^{L}|u_t|^{2}dx-\mathbb{E}_{p(x)}^{\alpha,\beta;\psi}(T),\;\; 0<t<T.\notag\\
\end{eqnarray}

Furthermore, on the left side of the inequality (\ref{eq12}), to get{\color{red}
\begin{align}
\phi(t)+\mathbb{E}_{p(x)}^{\alpha,\beta;\psi}(T)\big(T+\varepsilon e^{-\frac{T}{\varepsilon}}-\varepsilon\big)&\leq \varepsilon \phi'(0)(1-e^{-\frac{T}{\varepsilon}})+\phi(0)+{\mathcal{I}}_{p(x)}^{\alpha,\beta;\psi}(u^{*}(\cdot))\big(T+\varepsilon e^{-\frac{T}{\varepsilon}}-\varepsilon\big)\nonumber\\
&+\frac{3}{2}\int_{0}^{T}\int_{0}^{t}\int_{0}^{L}|u_t|^{2}e^{\frac{\tau -t}{\varepsilon}}dxd\tau dt.\label{eq13}
\end{align}}

{\bf Affirmation:} $\dfrac{3}{2}\displaystyle\int_{0}^{T}\displaystyle\int_{0}^{t}\displaystyle\int_{0}^{L}|u_t|^{2}e^{\frac{\tau -t}{\varepsilon}}dxd\tau dt \leq \frac{3 \Theta_{0}\varepsilon}{4}$.

Note that,
\begin{equation*}
    \{(\tau,t):0\leq t\leq T,\; 0\leq \tau \leq t\}=\{(\tau,t):0\leq \tau \leq T,\; \tau \leq t \leq T\}.
\end{equation*}

Thus, for any $T>0$, one has
\begin{align}
\frac{3}{2}\int_{0}^{T}\int_{0}^{t}\int_{0}^{L}|u_t|^{2}e^{\frac{\tau -t}{\varepsilon}}dxd\tau dt&=\frac{3}{2}\int_{0}^{T}\int_{\tau}^{T}e^{\frac{\tau -t}{\varepsilon}}dt\int_{0}^{L}|u_\tau|^{2}dxd\tau\nonumber\\
&=\frac{3\varepsilon}{2}\int_{0}^{T}\int_{0}^{L}|u_\tau|^{2}dx\big(1-e^{\frac{\tau -T}{\varepsilon}}\big)d\tau \nonumber\\
&\leq \frac{3\varepsilon}{2}\int_{0}^{T}\int_{0}^{L}|u_\tau|^{2}dxd\tau\nonumber\\
&\leq \frac{3\Theta_{0}\varepsilon}{4}.\label{eq14}
\end{align}

Now, we multiply the Eq.\eqref{eq11} by $\big(T+\varepsilon e^{-\frac{T}{\varepsilon}}-\varepsilon\big)^{-1}$ and let's not consider some nonnegative terms. In this sense, yields
\begin{equation}\label{eq15}
\mathbb{E}_{p(x)}^{\alpha,\beta;\psi}(T)\leq \frac{\varepsilon \phi'(0)(1-e^{\frac{-T}{\varepsilon}})+\phi(0)}{T+\varepsilon e^{\frac{-T}{\varepsilon}}-\varepsilon}+{\mathcal{I}}_{p(x)}^{\alpha,\beta;\psi}(u^{*}(\cdot))+\frac{\frac{3\Theta_{0}\varepsilon}{4}}{T+\varepsilon e^{\frac{-T}{\varepsilon}}-\varepsilon},\;\; T>\varepsilon.
\end{equation}

In addition, we use the fact that
$${\mathcal{I}}_{p(x)}^{\alpha,\beta;\psi}(u(\cdot,T))=\mathbb{E}_{p(x)}^{\alpha,\beta;\psi}(T)-\frac{\varepsilon}{2}\left\|u_{t}\right\|_{2}^{2}\leq \mathbb{E}_{p(x)}^{\alpha,\beta;\psi}(T),$$
yields
\begin{equation}\label{eq16}
\limsup_{T\to \infty+}{\mathcal{I}}_{p(x)}^{\alpha,\beta;\psi}(u(\cdot,T))\leq \limsup_{T\to \infty+} \mathbb{E}_{p(x)}^{\alpha,\beta;\psi}(T)\leq {\mathcal{I}}_{p(x)}^{\alpha,\beta;\psi}(u^{*}(\cdot)).
\end{equation}
Our deserved results are the consequence of the Eq.\eqref{eq10} and Eq.\eqref{eq16}
\end{proof}

From {\bf Lemma \ref{lemma2}}, it is possible to establish better estimates for $\|u_t\|_{2}^{2}$ and $\left\|{}^{\bf H}\mathbf{D}^{\alpha,\beta;\psi}_{0^{+}}u\right\|_{p(\cdot)}$
with $0\leq\beta\leq1$ and $\frac{1}{p(x)}<\alpha<1.$

\begin{lemma}\label{lemma4}
Let $u$ and $u^{*}$ be solutions of problems \eqref{Prob1} and \eqref{eq9} respectively. Then
\begin{itemize}

\item[(i)] $\|u_{t}(T)\|_{L^{2}(\Omega)}\overset{T\to \infty+}{\longrightarrow}0,$

\item[(ii)]$\displaystyle\int_{0}^{L}\left|{}^{\bf H}\mathbf{D}^{\alpha,\beta;\psi}_{0^{+}}u(\cdot,T)-{}^{\bf H}\mathbf{D}^{\alpha,\beta;\psi}_{0^{+}}u^{*}(\cdot)\right|^{p(x)}dx\overset{T\to \infty+}{\longrightarrow}0$.
\end{itemize}
\end{lemma}

\begin{proof}
\begin{itemize}

\item[(i)] Note that, the Eq.\eqref{eq7} shows that $\mathbb{E}_{p(x)}^{\alpha,\beta;\psi}(t)$ is bounded from below. Furthermore, we have that $\lim_{t\to \infty+ }\mathbb{E}_{p(x)}^{\alpha,\beta;\psi}(t)$ is well defined, (just use the fact of $\mathbb{E}_{p(x)}^{\alpha,\beta;\psi}(t)$ to be monotone and linear bounded), because every bounded monotone sequence converges. In this sense, using the {\bf Lemma \ref{lemma2}} and 
$\varepsilon \|u_t(\cdot,T)\|_{2}^{2}=2\,\,\mathbb{E}_{p(x)}^{\alpha,\beta;\psi}(T)-2\,\,{\mathcal{I}}_{p(x)}^{\alpha,\beta;\psi}(u(\cdot,T))$, we deduce that $\lim_{T\to \infty+}\|u_t(\cdot,T)\|_{2}^{2}$ exists.

Let us  define now $a=\lim_{T\to \infty+}\|u_t(\cdot,T)\|_{2}^{2}\geq 0$. Therefore $a=0$. In fact, if not, namely, $a\neq 0$, then $\exists 
 T_{}>0$ such that for $T\geq T_{0}$, yields
\begin{eqnarray*}
    \|u_t(\cdot,T)\|_{2}^{2}\geq \frac{a}{2}>0,
\end{eqnarray*}
that's implies
\begin{equation*}
\int_{T_{0}}^{T}\|u_t(\cdot,T)\|_{2}^{2}\geq \frac{a}{2}(T-T_{0}).
\end{equation*}

In this sense, we have a contradiction, since
$$\int_{0}^{T}\int_{0}^{L}|u_t|^{2}dxdt\leq \Theta_{0}.$$ 

\item[(ii)] Consider the following inequality 
\begin{eqnarray*}
    \langle|A|^{q-2}A-|B|^{q-2}B, A-B\rangle\geq 2^{2-q}|A-B|^{q},\;\; A,B\in \mathbb{R}^{N},\;\; q\geq 2.
\end{eqnarray*}

For $p(x)\geq 2$, we have
\begin{eqnarray}\label{eq17}
    &&\displaystyle\int_{0}^{L}\frac{1}{p(x)}\left|{}^{\bf H}\mathbf{D}^{\alpha,\beta;\psi}_{0^{+}}u\right|^{p(x)}dx-\displaystyle\int_{0}^{L}\frac{1}{p(x)}\left|{}^{\bf H}\mathbf{D}^{\alpha,\beta;\psi}_{0^{+}}u^{*}\right|^{p(x)}dx\notag\\
    &=&\displaystyle\int_{0}^{L}\frac{1}{p(x)}\displaystyle\int_{0}^{1}\left|\theta\;\; {}^{\bf H}\mathbf{D}^{\alpha,\beta;\psi}_{0^{+}}u+(1-\theta)\;\;{}^{\bf H}\mathbf{D}^{\alpha,\beta;\psi}_{0^{+}}u^{*}\right|^{p(x)}dxd\theta\nonumber\\
&=&{\color{red}\displaystyle\int_{0}^{L}\frac{1}{p(x)}\displaystyle\int_{0}^{1}\big|\theta\;\; {}^{\bf H}\mathbf{D}^{\alpha,\beta;\psi}_{0^{+}}u+(1-\theta)\;\;{}^{\bf H}\mathbf{D}^{\alpha,\beta;\psi}_{0^{+}}u^{*}\big|^{p(x)-2}}\nonumber\\
&&{\color{red}{}\big(\theta\;\; {}^{\bf H}\mathbf{D}^{\alpha,\beta;\psi}_{0^{+}}u+(1-\theta)\;\;{}^{\bf H}\mathbf{D}^{\alpha,\beta;\psi}_{0^{+}}u^{*}\big) \left( ^{H}{\bf D}^{\alpha,\beta;\psi}_{0+} u\,- \, ^{H}{\bf D}^{\alpha,\beta;\psi}_{0+} u^{*} \right)  dxd\theta}\nonumber\\
&\geq& \displaystyle\int_{0}^{L}\frac{1}{p(x)}\bigg(\int_{0}^{1}\frac{2^{2-p(x)}}{\theta}\big|\theta\;\; {}^{\bf H}\mathbf{D}^{\alpha,\beta;\psi}_{0^{+}}u+(1-\theta)\;\;{}^{\bf H}\mathbf{D}^{\alpha,\beta;\psi}_{0^{+}}u^{*}\nonumber\\
&-&{}^{\bf H}\mathbf{D}^{\alpha,\beta;\psi}_{0^{+}}u^{*}\big|^{p(x)}+\left|{}^{\bf H}\mathbf{D}^{\alpha,\beta;\psi}_{0^{+}}u^{*}\right|^{p(x)-2}\,\,{}^{\bf H}\mathbf{D}^{\alpha,\beta;\psi}_{0^{+}}u^{*} {}^{\bf H}\mathbf{D}^{\alpha,\beta;\psi}_{0^{+}}(u-u^{*})d\theta\bigg)dx\nonumber\\
&\geq& \displaystyle\int_{0}^{L}\displaystyle\int_{0}^{1}\theta^{p(x)-1}2^{2-p(x)}\left|{}^{\bf H}\mathbf{D}^{\alpha,\beta;\psi}_{0^{+}}u-{}^{\bf H}\mathbf{D}^{\alpha,\beta;\psi}_{0^{+}}u^{*}\right|^{p(x)}dxd\theta\nonumber\\
&+&\displaystyle\int_{0}^{L}\left|{}^{\bf H}\mathbf{D}^{\alpha,\beta;\psi}_{0^{+}}u^{*}\right|^{p(x)-2}{}^{\bf H}\mathbf{D}^{\alpha,\beta;\psi}_{0^{+}}u^{*}{}^{\bf H}\mathbf{D}^{\alpha,\beta;\psi}_{0^{+}}(u-u^{*})dx\nonumber\\
&=&\frac{2^{2-p^{+}}}{p^{+}}\displaystyle\int_{0}^{L}\left|{}^{\bf H}\mathbf{D}^{\alpha,\beta;\psi}_{0^{+}}u-{}^{\bf H}\mathbf{D}^{\alpha,\beta;\psi}_{0^{+}}u^{*}\right|^{p(x)}dx+\displaystyle\int_{0}^{L}g(x)(u-u^{*})dx,
\end{eqnarray}
where $u^{*}$ is the solution of the problem (\ref{eq9}). Note that, on the right side of the inequality (\ref{eq17})
\begin{equation}\label{eq18}
\frac{2^{2-p^{+}}}{p^{+}}\int_{0}^{L}\left|{}^{\bf H}\mathbf{D}^{\alpha,\beta;\psi}_{0^{+}}u-{}^{\bf H}\mathbf{D}^{\alpha,\beta;\psi}_{0^{+}}u^{*}\right|^{p(x)}dx\leq {\mathcal{I}}_{p(x)}^{\alpha,\beta;\psi}(u)-{\mathcal{I}}_{p(x)}^{\alpha,\beta;\psi}(u^{*}).
\end{equation}
So the result is follows from the {\bf Lemma \ref{lemma2}}.
\end{itemize}
\end{proof}

Now, let's investigate the main contribution of the present paper, i.e., the proof of {\bf Theorem \ref{theorem 1}}

\begin{proof} (\textbf{Proof of Theorem \ref{theorem 1}}) {\color{red} Consider $w=u-u^{*}$. The proof of the theorem will be divided into three steps.}

\textbf{Step1}: {\bf The purpose of this step is to discuss the sign of the error functional.}

Let {\color{red}
\begin{eqnarray*}
    \mathcal{G}(t)=\int_{0}^{L}\left(\frac{1}{2\varepsilon}w^{2}+ww_{t}+\varepsilon w_{t}^{2}\right)dx+2{\mathcal{I}}_{p(x)}^{\alpha,\beta;\psi}(u)- 2 {\mathcal{I}}_{p(x)}^{\alpha,\beta;\psi}(u^{*}).
\end{eqnarray*}}

Using the problem (\ref{Prob1}), yields
\begin{eqnarray}\label{eq19}
\mathcal{G}'(t)&=&\frac{1}{\varepsilon}\int_{0}^{L}\left(g(x)w-\left|{}^{\bf H}\mathbf{D}^{\alpha,\beta;\psi}_{0^{+}}u\right|^{p(x)-2}\;\;{}^{\bf H}\mathbf{D}^{\alpha,\beta;\psi}_{0^{+}}u\;\;{}^{\bf H}\mathbf{D}^{\alpha,\beta;\psi}_{0^{+}}w\right)dx- {\color{red}\int_{0}^{L}|w_t|^{2}dx}\nonumber\\
&=&\frac{1}{\varepsilon}\int_{0}^{L}\left(\left|{}^{\bf H}\mathbf{D}^{\alpha,\beta;\psi}_{0^{+}}u^{*}\right|^{p(x)-2}\;\;{}^{\bf H}\mathbf{D}^{\alpha,\beta;\psi}_{0^{+}}u^{*}\;\;{}^{\bf H}\mathbf{D}^{\alpha,\beta;\psi}_{0^{+}}w-\left|{}^{\bf H}\mathbf{D}^{\alpha,\beta;\psi}_{0^{+}}u\right|^{p(x)-2}\,\,{}^{\bf H}\mathbf{D}^{\alpha,\beta;\psi}_{0^{+}}u\;\;{}^{\bf H}\mathbf{D}^{\alpha,\beta;\psi}_{0^{+}}w\right)dx\notag\\&-&{\color{red}\int_{0}^{L}|w_{t}|^{2}dx}\nonumber\\
&\leq& -\frac{2^{2-p^{+}}}{\varepsilon}\int_{0}^{L}\left|{}^{\bf H}\mathbf{D}^{\alpha,\beta;\psi}_{0^{+}}w\right|^{p(x)}dx-{\color{red}\int_{0}^{L}|w_{t}|^{2}dx\leq 0,\;\; t>0.}\notag\\ 
\end{eqnarray}

\textbf{Step2}: The objective of this steps is to establish a relationship between $\mathcal{G}(t)$, $\displaystyle\int_{0}^{L}\left|{}^{\bf H}\mathbf{D}^{\alpha,\beta;\psi}_{0^{+}}w\right|^{p(x)}dx$ and {\color{red}$\displaystyle\int_{0}^{L}|w_{t}|^{2}dx$.} 

First, we need to prove that $\big\|{}^{\bf H}\mathbf{D}^{\alpha,\beta;\psi}_{0^{+}}u^{*}\big\|_{p(\cdot)}$ is bounded from above. 

Indeed, by multiplying the first identity of the problem \eqref{eq9} by $u^{*}$, integrate over $\Omega$, using the Poincar\'e inequality and the H\"older inequality, one has
\begin{eqnarray*}
    \int_{0}^{L}\left|{}^{\bf H}\mathbf{D}^{\alpha,\beta;\psi}_{0^{+}}u^{*}\right|^{p(x)}dx=\int_{0}^{L}g(x)u^{*}dx\Rightarrow \int_{0}^{L}\left|{}^{\bf H}\mathbf{D}^{\alpha,\beta;\psi}_{0^{+}}u^{*}\right|^{p(x)}dx\leq C\max\left\{\|g\|_{p'(\cdot)}^{p'_{+}},\|g\|_{p'(\cdot)}^{p'_{-}}\right\}\leq \Theta_{0}.
\end{eqnarray*}
Using the Eq.\eqref{eq6}, we get
\begin{equation}\label{eq20}
\int_{0}^{L}\left|{}^{\bf H}\mathbf{D}^{\alpha,\beta;\psi}_{0^{+}}w\right|^{p(x)}dx\leq 2^{p^{+}-1}\left(\int_{0}^{L}\big|{}^{\bf H}\mathbf{D}^{\alpha,\beta;\psi}_{0^{+}}u\big|^{p(x)}dx+\int_{0}^{L}|{}^{\bf H}\mathbf{D}^{\alpha,\beta;\psi}_{0^{+}}u^{*}\big|^{p(x)}dx\right)\leq 2^{p^{+}}\Theta_{0}=\Theta_{1}.
\end{equation}
In forthcoming proof, we first estimate $\mathcal{G}(t)$. By means of the H\"older inequality, we have
\begin{align}
\mathcal{G}(t)&\leq \frac{\varepsilon +1}{2\varepsilon}\int_{0}^{L}|w|^{2}dx+\frac{2\varepsilon +1}{2\varepsilon}\int_{0}^{L}|w_{t}|^{2}dx +2{\mathcal{I}}_{p(x)}^{\alpha,\beta;\psi}(u)-2{\mathcal{I}}_{p(x)}^{\alpha,\beta;\psi}(u^{*})\nonumber\\
&=\frac{\varepsilon +1}{2\varepsilon}\int_{0}^{L}|w|^{2}dx+\frac{2\varepsilon +1}{2\varepsilon}\int_{0}^{L}|w_{t}|^{2}dx {\color{red}+2\int_{0}^{L}g(x)w dx+\int_{0}^{L}\frac{1}{p(x)}\big|{}^{\bf H}\mathbf{D}^{\alpha,\beta;\psi}_{0^{+}}u\big|^{p(x)}dx}\nonumber\\
&-\int_{0}^{L}\frac{1}{p(x)}\big|{}^{\bf H}\mathbf{D}^{\alpha,\beta;\psi}_{0^{+}}u^{*}\big|^{p(x)}dx:=\mathcal{A}_{1}+\mathcal{A}_{2}+\mathcal{A}_{3}+\mathcal{A}_{4}.\label{eq21}
\end{align}

Using the inequality (\ref{A1}) and Hölder inequality, it follows that
\begin{align*}
\mathcal{A}_{1}&\leq \frac{(\varepsilon +1)}{2\varepsilon}C_{3}\max \left\{\left(\int_{0}^{L}\left|{}^{\bf H}\mathbf{D}^{\alpha,\beta;\psi}_{0^{+}}w\right|^{p(x)}dx\right)^{\frac{2}{p^{+}}}, \left(\int_{0}^{L}\left|{}^{\bf H}\mathbf{D}^{\alpha,\beta;\psi}_{0^{+}}w\right|^{p(x)}dx\right)^{\frac{2}{p^{-}}} \right\}\\
&\leq \frac{(\varepsilon +1)}{2\varepsilon}C_{3}\max\left\{\Theta_{1}^{\frac{2}{p^{-}}-\frac{1}{p^{+}}},\Theta_{1}^{\frac{1}{p^{+}}}\right\}\left(\int_{0}^{L}\left|{}^{\bf H}\mathbf{D}^{\alpha,\beta;\psi}_{0^{+}}w\right|^{p(x)}dx\right)^{\frac{1}{p^{+}}}.
\end{align*}

Furthermore, one has
\begin{align}\label{eq22}
\mathcal{A}_{3}&\leq 2C_{3}\|g\|_{p'(x)}\max \left\{\left(\int_{0}^{L}\left|{}^{\bf H}\mathbf{D}^{\alpha,\beta;\psi}_{0^{+}}w\right|^{p(x)}dx\right)^{\frac{1}{p^{+}}}, \left(\int_{0}^{L}\left|{}^{\bf H}\mathbf{D}^{\alpha,\beta;\psi}_{0^{+}}w\right|^{p(x)}dx\right)^{\frac{1}{p^{-}}} \right\}\nonumber\\
&\leq 2C_{3}\|g\|_{p'(x)}\max\left\{\Theta_{1}^{\frac{1}{p^{-}}-\frac{1}{p^{+}}},1 \right\}\left(\int_{0}^{L}\left|{}^{\bf H}\mathbf{D}^{\alpha,\beta;\psi}_{0^{+}}w\right|^{p(x)}dx\right)^{\frac{1}{p^{+}}}.
\end{align}

In the inequality above, we use
$$\left(\int_{0}^{L}\left|{}^{\bf H}\mathbf{D}^{\alpha,\beta;\psi}_{0^{+}}w\right|^{p(x)}dx\right)^{\eta}\leq \max\left\{\Theta_{1}^{\eta - \delta},1 \right\}\left(\int_{0}^{L}\left|{}^{\bf H}\mathbf{D}^{\alpha,\beta;\psi}_{0^{+}}w\right|^{p(x)}dx\right)^{\delta},\;\; 0<\delta \leq \eta.$$

On the other hand, by using the following inequality
\begin{eqnarray*}
    &&\left|\frac{\left|{}^{\bf H}\mathbf{D}^{\alpha,\beta;\psi}_{0^{+}}u\right|^{p(x)}}{p(x)} - \frac{\left|{}^{\bf H}\mathbf{D}^{\alpha,\beta;\psi}_{0^{+}}u^{*}\right|^{p(x)}}{p(x)}\right|\notag\\&\leq& \left(\left|{}^{\bf H}\mathbf{D}^{\alpha,\beta;\psi}_{0^{+}}u\right|+\left|{}^{\bf H}\mathbf{D}^{\alpha,\beta;\psi}_{0^{+}}u^{*}\right|\right)^{p(x)-1}\left|{}^{\bf H}\mathbf{D}^{\alpha,\beta;\psi}_{0^{+}}u- {}^{\bf H}\mathbf{D}^{\alpha,\beta;\psi}_{0^{+}}u^{*}\right|,
\end{eqnarray*}
we obtain
\begin{align}\label{eq23}
\mathcal{A}_{4}&\leq {\color{red}\int_{0}^{L}\left(\left|{}^{\bf H}\mathbf{D}^{\alpha,\beta;\psi}_{0^{+}}u\right|+\left|{}^{\bf H}\mathbf{D}^{\alpha,\beta;\psi}_{0^{+}}u^{*}\right|\right)^{p(x)-1}\left|{}^{\bf H}\mathbf{D}^{\alpha,\beta;\psi}_{0^{+}}w\right|dx}\nonumber\\
&\leq C_{5}(\Theta_{0})\max\left\{\Theta_{1}^{\frac{1}{p^{-}}-\frac{1}{p^{+}}},1 \right\}\left(\int_{0}^{L}\left|{}^{\bf H}\mathbf{D}^{\alpha,\beta;\psi}_{0^{+}}w\right|^{p(x)}dx\right)^{\frac{1}{p^{+}}}.
\end{align}

Using the Eq.\eqref{eq21} and Eq.\eqref{eq23}, yields
\begin{align}\label{eq24}
\mathcal{G}(t)&\leq  \frac{(\varepsilon +1)C_{3}}{2\varepsilon}\max\left\{\Theta_{1}^{\frac{2}{p^{-}}-\frac{2}{p^{+}}},1\right\}\left(\int_{0}^{L}\left|{}^{\bf H}\mathbf{D}^{\alpha,\beta;\psi}_{0^{+}}w\right|^{p(x)}dx\right)^{\frac{2}{p^{+}}}+\frac{2\varepsilon +1}{2}\int_{0}^{L}|u_t|^{2}dx\nonumber\\
&+C_{3}\|g\|_{p'(x)}\max\left\{\Theta_{1}^{\frac{1}{p^{-}}-\frac{1}{p^{+}}},1 \right\}\left(\int_{0}^{L}\left|{}^{\bf H}\mathbf{D}^{\alpha,\beta;\psi}_{0^{+}}w\right|^{p(x)}dx\right)^{\frac{1}{p^{+}}}\notag\\&{\color{red}+C_{5}(\Theta_{0})\max\left\{\Theta_{1}^{\frac{1}{p^{-}}-\frac{1}{p^{+}}},1 \right\}\left(\int_{0}^{L}\left|{}^{\bf H}\mathbf{D}^{\alpha,\beta;\psi}_{0^{+}}u\right|^{p(x)}dx\right)^{\frac{1}{p^{+}}}}\nonumber\\
&\leq \tilde{C}\left(\int_{0}^{L}\left|{}^{\bf H}\mathbf{D}^{\alpha,\beta;\psi}_{0^{+}}w\right|^{p(x)}dx\right)^{\frac{1}{p^{+}}}+\frac{2\varepsilon +1}{2}\|u_t\|_{2}^{2}, 
\end{align}
where 
\begin{eqnarray*}
    \tilde{C}&=&\frac{(\varepsilon +1)C_{3}}{2\varepsilon}\max\left\{\Theta_{1}^{\frac{2}{p^{-}}-\frac{1}{p^{+}}}, \Theta_{1}^{\frac{1}{p^{+}}}\right\}+2C_{3}\|g\|_{p'(x)}\max\left\{\Theta_{1}^{\frac{1}{p^{-}}-\frac{1}{p^{+}}},1 \right\}\notag\\&+&C_{3}(\Theta_{1})\max\left\{\Theta_{1}^{\frac{1}{p^{-}}-\frac{1}{p^{+}}},1 \right\}.
\end{eqnarray*}

So, from the Eq.\eqref{eq19} and Eq.\eqref{eq24}, we obtain
\begin{equation}\label{eq25}
\mathcal{G}(t)\leq \tilde{C}\left(\frac{-\varepsilon \mathcal{G}'(t)}{2^{2-p^{+}}}\right)^{\frac{1}{p^{+}}}-\mathcal{G}'(t).
\end{equation}

\textbf{Step 3}: In this step, we establish the first-order differential inequality. From Eq.\eqref{eq19}, yields
\begin{equation}\label{eq26}
{\color{red}|\mathcal{G}'(t)|\leq \frac{C_{3}}{\varepsilon}\|g\|_{p'(x)}\left\|{}^{\bf H}\mathbf{D}^{\alpha,\beta;\psi}_{0^{+}}w\right\|_{p(x)}+\dfrac{1}{\varepsilon}\left\|\left|{}^{\bf H}\mathbf{D}^{\alpha,\beta;\psi}_{0^{+}}u\right|^{p(x)-1}\right\|_{p'(x)}\left\|{}^{\bf H}\mathbf{D}^{\alpha,\beta;\psi}_{0^{+}}w\right\|_{p(x)}+\int_{0}^{L}|w_t|^{2}dx\leq \Theta_{2},}
\end{equation}
where 
\begin{eqnarray*}
    \Theta_{2}=\frac{C_{3}}{\varepsilon}\|g\|_{p'(x)}\max\left\{\Theta_{1}^{\frac{1}{p^{-}}},\Theta_{1}^{\frac{1}{p^{+}}}\right\}+\frac{1}{\varepsilon}\max\left\{\Theta_{1}^{\frac{p^{-}-1}{p^{-}}},\Theta_{1}^{\frac{p^{+}-1}{p^{+}}}\right\}\max\left\{\Theta_{1}^{\frac{1}{p^{-}}},\Theta_{1}^{\frac{1}{p^{+}}}\right\}+\frac{\Theta_{0}}{\varepsilon}.
\end{eqnarray*}

In this sense, using the inequality \eqref{eq26}, the Eq.\eqref{eq25} can be rewritten
\begin{equation}\label{eq27}
\mathcal{G}(t)\leq \left(\left(\frac{\varepsilon(\tilde{C})^{p^{+}}}{2^{2-p^{+}}}\right)^{\frac{1}{p^{+}}}+\Theta_{2}^{1-\frac{1}{p^{+}}}\right)\big(-\mathcal{G}'(t)\big)^{\frac{1}{p^{+}}}:=C_{4}\big(-\mathcal{G}'(t)\big)^{\frac{1}{p^{+}}}.
\end{equation}
To solve the Eq.\eqref{eq27}, we have
\begin{equation}\label{*}
    \mathcal{G}(t)\leq \big(\mathcal{G}^{1-p^{+}}(0)C_{4}^{-p^{+}}(p^{+}-1)t\big)^{\frac{1}{1-p^{+}}}\leq \big(C_{4}^{-p^{+}}(p^{+}-1)\big)^{\frac{1}{1-p^{+}}}t^{\frac{1}{1-p^{+}}}.
\end{equation}

The inequality (\ref{*}) together with the non negativity of the form 
\begin{eqnarray*}
\frac{1}{2\varepsilon}w^{2}+ww_{t}+\varepsilon w_{t}^{2}=\left(\frac{1}{2\sqrt{\varepsilon}}w^{2}+\sqrt{\varepsilon}w_{t}\right)^{2}+\frac{1}{4\varepsilon}w^{2},
\end{eqnarray*}
we obtain
\begin{eqnarray*}
    {\mathcal{I}}_{p(x)}^{\alpha,\beta;\psi}(u)-{\mathcal{I}}_{p(x)}^{\alpha,\beta;\psi}(u^{*})\leq \frac{1}{2}\left(C_{4}^{-p^{+}}(p^{+}-1)\right)^{\frac{1}{1-p^{+}}}t^{\frac{1}{1-p^{+}}}.
\end{eqnarray*}

Finally, from Eq.\eqref{eq18}, one has
\begin{eqnarray*}
    \int_{0}^{L}\left|{}^{\bf H}\mathbf{D}^{\alpha,\beta;\psi}_{0^{+}}u^{*}-\,\,{}^{\bf H}\mathbf{D}^{\alpha,\beta;\psi}_{0^{+}}u(x,t)\right|^{p(x)}dx\leq \frac{p^{+}}{2^{3-p^{+}}}\left(C_{}^{-p^{+}}(p^{+}-1)\right)^{\frac{1}{1-p^{+}}}t^{\frac{1}{1-p^{+}}}=\Theta t^{\frac{1}{1-p^{+}}},
\end{eqnarray*}
where {\color{red}
\begin{equation}\label{eq28}
\Theta=\big(C_{4}^{-p^{+}}(p^{+}-1)\big)^{\frac{1}{1-p^{+}}},\;\; C_{4}:=\left(\left(\frac{\varepsilon(\tilde{C})^{p^{+}}}{2^{2-p^{+}}}\right)^{\frac{1}{p^{+}}}+\Theta_{2}^{1-\frac{1}{p^{+}}}\right).
\end{equation}}

Therefore, we conclude the proof.
\end{proof}

\section*{Declarations}

\subsection*{Funding:} This research received no external funding.

\subsection*{Data Availability Statement:} Data sharing not applicable to this paper as no data sets were generated or analyzed during the current study.

\subsection*{Acknowledgments:} The authors are very grateful to the anonymous reviewers for their
useful comments that led to improvement of the manuscript.

\subsection*{Conflicts of Interest:} The authors declare no conflict of interest.



\begin{thebibliography}{}
	
	
\bibitem{brezis} Brezis, H. Functional Analysis, Sobolev Spaces and Partial Differential Equations, 1st. Edition, Springer, 2010.

\bibitem{minrad}Mingione,G., R\v{a}dulescu,V., Recent developments in problems with nonstandard growth and nonuniform ellipticity. J. Math. Anal. Appl. 501, 125197 (2021)




\bibitem{fzz}Fan,X., Zhang,Q. ,Zhao,Y., A strong maximum principle for $p(x)-$Laplace equations. Chin. J. Contemp. Math. 21, 1–7 (2000)


\bibitem{acermingi}Acerbi, E., Mingione,G., Regularity results for stationaryelectro-rheological fluids. Arch. Ration. Mech.Anal.164, 213–259 (2002)

\bibitem{ruzi}Ruzicka,M., Electrorheological Fluids: Modeling and Mathematical Theory. Springer,Berlin (2000)



\bibitem{rare}R\v{a}dulescu,V.,Repov\v{s},D., Partial Differential Equations with Variable Exponents:Variational Methods and Qualitative Analysis. CRC Press, Boca Raton (2015)

\bibitem{rad}R\v{a}dulescu,V.D., Nonlinear elliptic equations with variable exponent: old and new. NonlinearAnal. 121, 336–369 (2015)

\bibitem{winslow}Winslow,W., Induced fibration of suspensions. J.Appl.Phys. 20, 1137–1140(1949)

\bibitem{13} Wei, D. M. Nonlinear wave equations arising in modeling of some strain-hardening structures, in: Proceedings of joint Conference of the 6th International Conference of Computational Physics and 2004 Conference of Computational Physics, Beijing, China, 23-28 May, 2004.

\bibitem{Bozorgnia} Bozorgnia, F., and P. Lewintan. Decay estimates for solutions of evolutionary damped $p$-Laplace equations. Electron. J. Diff. Equ., 2021 (73) (2021), pp. 1-9.

\bibitem{3} Baravdish, G., O. Svensson, M. Gulliksson, and Y. Zhang. Damped second order flow applied to image denoising. IMA J. Appl. Math. 84.6 (2019): 1082-1111.

\bibitem{6} Frankel, S. P., Convergence rates of iterative treatments of partial differential equations. Math. Tables and Other Aids to Comput. 4, (1950), pp. 65 – 75.

\bibitem{10}  Polyak, B. T., Some methods of speeding up the convergence of iteration methods, Zh. Vychisl. Mat. Mat. Fiz. 4(1964), no. 5, pp. 791-803 (Russian); engl. trans. in U.S.S.R. Comput. Math. Math. Phys. 4 (1964), no. 5, pp. 1-17.

\bibitem{Li} Li, F., and X. Zhu. Convergence rate of solutions to the generalized telegraph equation with an inhomogeneous force. J. Math. Anal. Appl. 517.1 (2023): 126564.

\bibitem{Yuksekkaya} Yuksekkaya, H., E. Piskin, J. Ferreira, and M. Shahrouzi. A viscoelastic wave equation with delay and variable exponents: existence and nonexistence. Z. Angew. Math. Phys. 73.4 (2022): 133.

\bibitem{Shahrouzi} Shahrouzi, M.. Global nonexistence of solutions for a class of viscoelastic Lamé system. Indian J. Pure Appl. Math. 51 (2020): 1383-1397.

\bibitem{Messoaudi2} Messoaudi, S., M. Al-Gharabli, and A. Al-Mahdi. On the existence and decay of a viscoelastic system with variable-exponent nonlinearity. Disc. Cont. Dyn. Sys.-S (2022): 0-0.


\bibitem{Guo} Guo, B., and F. Liu. A lower bound for the blow-up time to a viscoelastic hyperbolic equation with nonlinear sources. Appl. Math. Lett. 60 (2016): 115-119.


\bibitem{Li2} Li, F., and F. Liu. Blow-up of solutions to a quasilinear wave equation for high initial energy. Comptes Rendus Mecanique 346.5 (2018): 402-407.

\bibitem{Yang1} Yang, H., and Y. Han. Blow-up for a damped $p$-Laplacian type wave equation with logarithmic nonlinearity. J. Diff. Equ. 306 (2022): 569-589.

\bibitem{Majee} Majee, S., S. K. Jain, R. K. Ray, and A. K. Majee. On the development of a coupled nonlinear telegraph-diffusion model for image restoration. Comput. Math. Appl. 80.7 (2020): 1745-1766.

\bibitem{Pei1} Pei, P., M. A. Rammaha, and D. Toundykov. Weak solutions and blow-up for wave equations of $p$-Laplacian type with supercritical sources. J. Math. Phys. 56.8 (2015): 081503.

\bibitem{Antontsev} Antontsev, S. Wave equation with $p(x,t)$-Laplacian and damping term: blow-up of solutions. Comptes Rendus Mécanique 339.12 (2011): 751-755.

\bibitem{Acerbi} Acerbi, E., and G. Mingione. Regularity results for stationary electro-rheological fluids. Arc. Rational Mech. Anal. 164 (2002): 213-259.




\bibitem{Boudjeriou1} Boudjeriou, T. Global existence and blow-up of solutions for a parabolic equation involving the fractional $p(x)$-Laplacian. Applicable Anal. 101.8 (2022): 2903-2921.

\bibitem{Boudjeriou2} Boudjeriou, T. On the diffusion $p(x)$-Laplacian with logarithmic nonlinearity. J. Ellip. Parabolic Equ. 6.2 (2020): 773-794.

\bibitem{Boudjeriou3} Boudjeriou, T. Global well‐posedness and finite time blow‐up for a class of wave equation involving fractional $p$-Laplacian with logarithmic nonlinearity. Mathematische Nach. 296.3 (2023): 938-956.

\bibitem{Liao} Liao, M., Q. Liu, and H. Ye. Global existence and blow-up of weak solutions for a class of fractional $p$-Laplacian evolution equations. Adv. Nonlinear Anal. 9.1 (2020): 1569-1591.

\bibitem{Jiang} Jiang, R., and J. Zhou. Blow-up and global existence of solutions to a parabolic equation associated with the fraction $p$-Laplacian. Commun. Pure Appl. Anal. 18.3 (2019): 1205-1226.


\bibitem{Stavrakakis} Stavrakakis, N. M., and A. N. Stylianou. Global attractor for some wave equations of $p$-and $p(x)$-Laplacian type. Diff. Integral Equ. 24(1-2): 159-176 (2011).


\bibitem{Messaoudi} Messaoudi, S. A., and A. A. Talahmeh. A blow-up result for a nonlinear wave equation with variable-exponent nonlinearities. Applicable Anal. 96.9 (2017): 1509-1515.

\bibitem{Messaoudi23} Messaoudi, S. A., Talahmeh, A. A. Blow up solutions of a quasilinear wave equation with variable exponent nonlinearities
Math. Meth. Appl. Sci., 40 (2017): 6976-6986.


\bibitem{Messaoudi1} Messaoudi, S. A., J. H. Al-Smail, and A. A. Talahmeh. Decay for solutions of a nonlinear damped wave equation with variable-exponent nonlinearities. Comput. Math. Appl. 76.8 (2018): 1863-1875.


\bibitem{passa} Bocea, M., and M. Mihailesc. On the continuity of the Luxemburg norm of the gradient in $L^{p(\cdot)}(\Omega)$ with respect to $p(\cdot)$. Proc. Amer. Math. Soc. 142.2 (2014): 507-517.

\bibitem{Diethelm} Diethelm, K., and N. J. Ford. Analysis of fractional differential equations. J. Math. Anal. Appl. 265.2 (2002): 229-248.

\bibitem{Lakshmikantham} Lakshmikantham, V., and A. S. Vatsala. Basic theory of fractional differential equations. Nonlinear Analysis: Theory, Methods \& Applications 69.8 (2008): 2677-2682.

\bibitem{Kilbas} Kilbas, A. A., Hari M. Srivastava, and Juan J. Trujillo. Theory and applications of fractional differential equations. Vol. 204. elsevier, 2006.




\bibitem{Alsaedi} Alsaedi, R., and A. Ghanmi. Variational approach for the Kirchhoff problem involving the $p$-Laplace operator and the $\psi$-Hilfer derivative. Math. Meth. Appl. Sci.

\bibitem{Sousa} Sousa, J. Vanterler da C. Existence and uniqueness of solutions for the fractional differential equations with $p$-Laplacian in $\mathbb {H}^{\nu,\eta;\psi}_{p}$. J. Appl. Anal. Comput. 12.2 (2022): 622-661.

\bibitem{Sousa12} Sousa, J. Vanterler da C., César T. Ledesma, Mariane Pigossi, and Jiabin Zuo. Nehari manifold for weighted singular fractional $p$-Laplace equations. Bull. Braz. Math. Soc. 53.4 (2022): 1245-1275.

\bibitem{Zuo} Zuo, Jiabin, Khaled Khachnaoui, and J. Vanterler da C. Sousa. Ground state solutions for electromagnetic Schrödinger equations on unbounded domains. Commun. Nonlinear Sci. Numer. Simul. 118 (2023): 107033.

\bibitem{Sousa} Sousa, J. Vanterler da C., Nemat Nyamoradi, and M. Lamine. "Nehari manifold and fractional Dirichlet boundary value problem." Anal. Math. Phys. 12.6 (2022): 143.

\bibitem{Sousa1} Sousa, J. Vanterler da C., Jiabin Zuo, and Donal O'Regan. The Nehari manifold for a $\psi$-Hilfer fractional $p$-Laplacian. Applicable Anal. 101.14 (2022): 5076-5106.

\bibitem{Ezati} Ezati, R., and N. Nyamoradi. Existence and multiplicity of solutions to a $\psi$-Hilfer fractional $p$-Laplacian equations. Asian-European J. Math. (2022): 2350045.

\bibitem{Fan} Fan, X., T. Xue, and Y. Jiang. Dirichlet problems of fractional $p$-Laplacian equation with impulsive effects. Math. Bios. Engine. 20.3 (2023): 5094-5116.


\bibitem{Ledesma} Ledesma, C. E. Torres, and N. Nyamoradi. $(k,\psi)$-Hilfer variational problem. J. Elliptic Parabolic Equ. 8.2 (2022): 681-709.

\bibitem{Ledesma1}  Ledesma, C. E. Torres, and N. Nyamoradi. $(k,\psi)$-Hilfer impulsive variational problem. Revista de la Real Academia de Ciencias Exactas, Físicas y Naturales. Serie A. Matemáticas 117.1 (2023): 42.

\bibitem{Srivastava} Srivastava, H. M., and J. Vanterler da C. Sousa. Multiplicity of Solutions for Fractional-Order Differential Equations via the $\kappa(x)$-Laplacian Operator and the Genus Theory. Fractal and Fractional 6.9 (2022): 481.


\bibitem{Sousa4} Sousa, J. Vanterler da C., and E. Capelas de Oliveira. On the $\psi$-Hilfer fractional derivative. Commun. Nonlinear Sci. Numer. Simul. 60 (2018): 72-91.

\bibitem{Sousa5} Sousa, J. Vanterler da C., and E. Capelas de Oliveira. Leibniz type rule: $\psi$-Hilfer fractional operator. Commun. Nonlinear Sci. Numer. Simul. 77 (2019): 305-311.


\end{thebibliography}
\end{document}